\documentclass[12pt]{amsart}
\usepackage{hyperref,amssymb,enumerate, enumitem}
\usepackage[capitalise, nameinlink]{cleveref}
\makeatletter
\@namedef{subjclassname@2010}{%
  \textup{2010} Mathematics Subject Classification}
\makeatother

\newtheorem{theorem}{Theorem}[section]
\newtheorem{prop}[theorem]{Proposition}
\newtheorem{corollary}[theorem]{Corollary}
\newtheorem{lemma}[theorem]{Lemma}
\newtheorem{mainthm}{Theorem}

\theoremstyle{definition}
\newtheorem{definition}[theorem]{Definition}
\newtheorem{rem}[theorem]{Remark}
\newtheorem{example}[theorem]{Example}

\def\R{{\mathbb{R}}}
\def\N{{\mathbb{N}}}
\def\C{{\mathbb{C}}}
\def\Q{{\mathbb{Q}}}
\def\Z{{\mathbb{Z}}}

\hypersetup{
colorlinks,
linkcolor=blue,          
filecolor=magenta,      
urlcolor=cyan           
}

\begin{document}
\title[Rokhlin Dimension and Equivariant Bundles]{Rokhlin Dimension and Equivariant Bundles}
\author[Prahlad Vaidyanathan]{Prahlad Vaidyanathan}
\address{Department of Mathematics\\ Indian Institute of Science Education and Research Bhopal\\ Bhopal ByPass Road, Bhauri, Bhopal 462066\\ Madhya Pradesh. India.}
\email{prahlad@iiserb.ac.in}
\date{}
\begin{abstract}
Given an action of a compact group on a complex vector bundle, there is an induced action of the group on the associated Cuntz-Pimsner algebra. We determine conditions under which this action has finite Rokhlin dimension.
\end{abstract}
\maketitle
\parindent 0pt

The study of group actions on C*-algebras has long been an important part of the general theory of C*-algebras, acting as an important source of new examples (via the crossed product construction), and as a way to understand C*-algebras intrinsically by analyzing their automorphisms. In this context, one is often interested in a useful noncommutative analogue of a \emph{free} action of a group on a topological space. There are a variety of different possible definitions of `freeness' in the context of C*-algebras (see, for instance, \cite{phillips}). In this paper, we wish to study one such notion, that of finite Rokhlin dimension. \\

Rokhlin dimension was invented by Hirshberg, Winter, and Zacharias \cite{hirshberg}, to provide a higher rank analogue of the Rokhlin property. While the original definition applied to finite groups (and the integers), Gardella \cite{gardella_compact} extended this definition to the case of a compact, second countable group. Finiteness of Rokhlin dimension is an important property for a group action to possess, as it implies (especially in the case of `commuting towers') certain structural properties of the crossed product. In particular, it allows one to show that many important properties pass from the underlying C*-algebra to the crossed product (see \cite{gardella_regularity}, \cite{gardella}). \\

In this paper, we wish to analyze the Rokhlin dimension of actions on Cuntz-Pimsner algebras arising from actions on the underlying C*-correspondence (the question of whether the Cuntz-Pimsner algebra has finite nuclear dimension has been asked, and partially addressed by Brown, Tikuisis and Zelenberg \cite{brown}). While it seems difficult to obtain a broad result covering all Cuntz-Pimsner algebras, we were able to do so when the underlying C*-correspondence arises from a vector bundle over a compact metric space. Our main result is as follows.

\begin{mainthm}\label{mainthm: action_bundle}
Let $G$ be a compact, second countable group, and $(E,p,X)$ be a locally trivial, complex, Hermitian vector bundle over a compact metric space $X$. Given an action of $G$ on $(E,p,X)$, let $\beta$ denote the induced action on the associated Cuntz-Pimsner algebra $\mathcal{O}_E$.
\begin{enumerate}[label=(\roman*)]
\item If the action of $G$ on $X$ is free, then
\[
\dim_{Rok}(\beta) \leq \dim(X/G).
\]
\item If $G$ is finite, the action of $G$ on $X$ is trivial, and the induced representation of $G$ on each fiber of $E$ is faithful, then
\[
\dim_{Rok}(\beta) \leq 2\dim(X)+1.
\]
\end{enumerate}
\end{mainthm}

Note that, for a compact Hausdorff space $Y$, $\dim(Y)$ refers to the Lebesgue covering dimension of $Y$, and for an action $\alpha$, $\dim_{Rok}(\alpha)$ refers to the Rokhlin dimension of $\alpha$. \\

The paper is organized as follows: In \cref{sec: rokhlin}, we review the notion of Rokhlin dimension for compact, second countable groups. We also prove a `local' version of the definition, while filling a small gap in the literature. In \cref{sec: cx_algebras}, we study $C(X)$-algebras, and prove that, under certain conditions, finiteness of Rokhlin dimension passes from the fibers of a $C(X)$-algebra to the ambient algebra (\cref{thm: rokhlin_cx_algebra}). In \cref{sec: correspondences}, we study group actions on C*-correspondences, and give a natural definition under which an action on a C*-correspondence induces an action of finite Rokhlin dimension on the associated Cuntz-Pimsner algebra. Finally, \cref{sec: bundles} is devoted to the proof of \cref{mainthm: action_bundle}.

\section{Rokhlin Dimension}\label{sec: rokhlin}

We begin by reviewing the notion of Rokhlin dimension from \cite{gardella_compact}. In what follows, all groups will be assumed to be compact and second countable, and all C*-algebras will be separable (unless otherwise stated). \\

Let $A$ be a C*-algebra. Let $\ell^{\infty}(\N,A)$ denote the C*-algebra of all bounded sequences in $A$ with the supremum norm, and pointwise operations. Write $c_0(\N,A)$ for those sequences in $\ell^{\infty}(\N,A)$ which converge to zero in the norm. Then, $c_0(\N,A)$ is an ideal of $\ell^{\infty}(\N,A)$. We write
\[
A_{\infty} := \frac{\ell^{\infty}(\N,A)}{c_0(\N,A)}
\]
for the corresponding quotient, and write $\eta_A : \ell^{\infty}(\N,A)\to A_{\infty}$ for the quotient map. Note that $A$ can be identified with the subalgebra of $\ell^{\infty}(\N,A)$ consisting of constant sequences. We identify $A$ with a subset of  $A_{\infty}$ via the composition $A\hookrightarrow \ell^{\infty}(\N,A)\xrightarrow{\eta_A} A_{\infty}$. We write $A_{\infty}\cap A'$ for the relative commutant of $A$ in $A_{\infty}$, and $\text{Ann}(A,A_{\infty})$ for the annihilator of $A$ in $A_{\infty}$. Then, $\text{Ann}(A,A_{\infty})$ is an ideal of $A_{\infty}\cap A'$, so we write
\[
F(A) := \frac{A_{\infty}\cap A'}{\text{Ann}(A,A_{\infty})}
\]
for the corresponding quotient, and $\kappa_A : A_{\infty}\cap A'\to F(A)$ for the natural quotient map. Note that $F(A)$ is unital if $A$ is $\sigma$-unital, where the unit $1_{F(A)}$ is the image of any countable approximate unit of $A$.\\

Let $G$ be a locally compact group, and let $A$ be a C*-algebra. By an action of $G$ on $A$, we mean a continuous group homomorphism $\alpha : G\to \text{Aut}(A)$, where $\text{Aut}(A)$ is equipped with the point norm topology. Given such an action, there are induced actions on $\ell^{\infty}(\N,A)$, on $A_{\infty}\cap A'$ and on $F(A)$, which we denote by $\alpha^{\infty}, \alpha_{\infty},$ and $F(\alpha)$ respectively. Note that these actions need not be continuous, so we write
\[
\ell_{\alpha}^{\infty}(\N,A) := \{a \in \ell^{\infty}(\N,A) : g\mapsto \alpha_g^{\infty}(a) \text{ is continuous}\}.
\]
We set $A_{\infty,\alpha} := \eta_A(\ell_{\alpha}^{\infty}(\N,A))$, then $A_{\infty,\alpha}$ is invariant under $\alpha_{\infty}$, and the restriction of $\alpha_{\infty}$ to this algebra (also denoted by $\alpha_{\infty}$) is continuous. Similarly, we set $F_{\alpha}(A) := \kappa_A(A_{\infty,\alpha}\cap A')$, and write $F(\alpha)$ for the continuous restriction of $F(\alpha)$ to $F_{\alpha}(A)$. \\

Let $A$ be a C*-algebra. We say that $a,b\in A$ are \emph{orthogonal} (in symbols, we write $a\perp b$) if $ab=ba = a^{\ast}b = b^{\ast}a = 0$. Given a contractive, completely positive map  $\varphi : A\to B$ between two C*-algebras, we say that $\varphi$ has \emph{order zero} if, for every $a,b\in A$, $\varphi(a) \perp \varphi(b)$ whenever $a\perp b$. As is customary, we will henceforth use the abbreviation \emph{c.c.p.} for `contractive and completely positive'. \\

Given a group $G$, we write $\sigma : G\to \text{Aut}(C(G))$ for the left action of $G$ on $C(G)$, given by $\sigma_s(f)(t) := f(s^{-1}t)$. The next definition is due to Gardella \cite{gardella_compact}, inspired by the case of a finite group, due to Hirshberg, Winter, and Zacharias \cite{hirshberg}.

\begin{definition}\cite[Definition 3.2]{gardella_compact}\label{defn: rokhlin_dim}
Let $G$ be a compact, second countable group, and $A$ be a separable C*-algebra. For an action $\alpha : G\to \text{Aut}(A)$, we say that $\alpha$ has Rokhlin dimension $d$ if $d$ is the least integer such that there exist equivariant, contractive, completely positive, order zero maps
\[
\varphi_0, \varphi_1, \ldots, \varphi_d : (C(G), \sigma) \to (F_{\alpha}(A), F(\alpha))
\]
such that $\sum_{j=0}^d \varphi_j(1) = 1$. We denote the Rokhlin dimension of $\alpha$ by $\dim_{Rok}(\alpha)$. If no such integer $d$ exists, we say that $\alpha$ has infinite Rokhlin dimension, and we write $\dim_{Rok}(\alpha) = +\infty$.
\end{definition}

\begin{rem}\label{rem: commuting_ranges}
In the above definition, if we require the maps $\{\varphi_j : 0\leq j\leq d\}$ to have commuting ranges, then one arrives at the definition of \emph{Rokhlin dimension with commuting towers}, and the corresponding integer is denoted by $\dim_{Rok}^c(\alpha)$. Clearly, the two notions coincide if the algebra $A$ is commutative, but they are, in general, different. \\

Some of our results (notably the first part of \cref{mainthm: action_bundle} - see \cref{thm: rokhlin_free}) holds true if we replace $\dim_{Rok}$ with the stronger notion of $\dim_{Rok}^c$. However, proving this does not involve significantly different ideas, so we have avoided this in favour of cleaner exposition.
\end{rem}

Our goal in this section is to give a `local' version of this definition. The theorem given below (\cref{lem: rokhlin_local}) is originally due to Gardella \cite[Lemma 3.7]{gardella_compact}. However, it seems to us that the original proof has a small gap. Specifically, in \cite[Lemma 3.7]{gardella_compact}, the author states that a c.c.p. map that preserves orthogonality on a dense subalgebra must have order zero. This assumes that the dense subalgebra has sufficiently many orthogonal elements, which may not necessarily be true. We now address this issue by constructing such a subalgebra in the situation which is of interest to us. \\

Given a compact metric space $(Y,d)$, we wish to construct a specific countable subalgebra $\mathcal{A}\subset C(Y)$ which is not only dense in $C(Y)$, but is also closed under multiplication by certain `bump' functions. To that end, we define
\[
\mathcal{F} := \{(K,U) : K\subset U\subset Y, K \text{ compact, and } U \text{ open}\}.
\]
For two elements $(K_1,U_2), (K_2,U_2) \in \mathcal{F}$, we define $(K_1,U_1)\leq (K_2,U_2)$ if $K_1\subset K_2$ and $U_2\subset U_1$. Note that this defines a partial order on $\mathcal{F}$. Now, fix a countable dense set $D := \{x_n\} \subset Y$. For each $n\in \N$ and $q\in \Q$ with $q>0$, define
\[
U_{n,q} := \{y\in Y: d(y,x_n) < q\} \text{ and } K_{n,q} := \{y\in Y : d(y,x_n) \leq q\}.
\]
For an integer $k\geq 1$, and tuples $\overline{n} = (n_1,n_2,\ldots, n_k) \in \N^k,$ and $\overline{q} = (q_1,q_2,\ldots, q_k) \in \Q^k$, we define
\[
U_{\overline{n},\overline{q}} := \bigcup_{j=1}^k U_{n_j,q_j} \text{ and } K_{\overline{n},\overline{q}} := \bigcup_{j=1}^k K_{n_j,q_j}.
\]
We set
\[
\mathcal{F}_0 := \bigcup_{k=1}^{\infty} \{(K_{\overline{n},\overline{r}}, U_{\overline{n},\overline{q}}) : \overline{n} \in \N^k, \overline{r}, \overline{q} \in \Q^k, r_i < q_i \text{ for all } 1\leq i\leq k\}.
\]
Note that $\mathcal{F}_0$ is a subset of $\mathcal{F}$ and is countable, because the set of all finite subsets of $\N\times \Q\times \Q$ is countable.

\begin{lemma}\label{lem: order_dominate}
Let $(Y,d)$ be a compact metric space and $\mathcal{F}$ and $\mathcal{F}_0$ be as above. Then, for any $(K,U) \in \mathcal{F}$, there exists $(K',U') \in \mathcal{F}_0$ such that $(K,U)\leq (K',U')$.
\end{lemma}
\begin{proof}
Fix $(K,U) \in \mathcal{F}$. We know that the collection $\{U_{n,q} : n \in \N, q\in \Q\}$ forms a basis for the metric topology on $Y$. Hence, we may write
\[
U = \bigcup_{(n,q) \in S} U_{n,q}
\]
for some subset $S\subset \N\times \Q$. Since $K$ is compact, there exist finitely many pairs $(n_1,q_1)$, $(n_2,q_2)$,$\ldots$, $(n_k,q_k)\in S$ such that
\[
K \subset \bigcup_{j=1}^k U_{n_j,q_j} =: U'.
\]
It follows that $U'\subset U$. We now construct $K'$ by induction on $k$. \\

If $k=1$: Then $K\subset U_{n_1,q_1}$. Consider the function $f:K\to \R$ given by $x\mapsto d(x,x_{n_1})$. This function satisfies $f(x) < q_1$ for all $x\in K$, and must attain a maximum on $K$. Hence, there exists $M \in \R$ such that $M < q_1$ and $f(x)\leq M$ for all $x\in K$. Choose a rational number $r_1\in \Q$ such that $M\leq r_1 < q_1$, and set $K' := K_{n_1,r_1}$. It follows that $(K',U') \in \mathcal{F}_0$ and $(K,U)\leq (K',U')$. \\

If $k\geq 2$, then assume, by induction, that we have chosen $r_i < q_i$ for $i=1,2,\ldots, k-1$ such that
\[
K \subset K_{n_1,r_1}\cup K_{n_2,r_2}\cup \ldots \cup K_{n_{k-1},r_{k-1}}\cup U_{n_k,q_k}
\]
and, for all $1\leq i\leq k-1$, $K\cap K_{n_i,q_i} \subset K_{n_i,r_i}$. Now consider $h:K\cap K_{n_k,q_k} \to \R$ given by $y\mapsto d(y,x_{n_k})$. Once again, this function must attain a maximum $M' < q_k$ on $K$. Choose a rational number $r_k \in \Q$ such that $M'\leq r_k < q_k$, so that $K\cap K_{n_k,q_k} \subset K_{n_k,r_k}$. Then,
\[
K\subset \bigcup_{j=1}^k K_{n_j,r_j} =: K'.
\]
So that $(K',U') \in \mathcal{F}_0$ and $(K,U)\leq (K',U')$.
\end{proof}

We now construct a countable subset $\mathcal{A} \subset C(Y)$ as follows: For each $(K,U) \in \mathcal{F}_0$, choose a function $\rho = \rho_{(K,U)} \in C(Y)$ such that $0\leq \rho\leq 1$ and
\[
\rho\lvert_{K} = 1 \text{ and } \rho\lvert_{Y\setminus U} = 0.
\]
For each pair of finite sets $F_1\subset \mathcal{F}_0$ and $F_2\subset D$, define a function $f_{F_1,F_2} \in C(Y)$ by
\[
f_{F_1,F_2}(x) = \left(\prod_{(K,U) \in F_1} \rho_{(K,U)}(x)\right)\left( \prod_{y\in F_2}d(x,y)\right).
\]
We allow the possibility that $F_1 = \emptyset$ or $F_2 = \emptyset$, in which case the corresponding product is simply the constant function $1$. Note that, for any other pair of finite sets $F_1' \subset \mathcal{F}_0$ and $F_2'\subset D$, we have $f_{F_1,F_2}f_{F_1',F_2'} = f_{F_1\cup F_1', F_2\cup F_2'}$. We define $\mathcal{A}$ to be the linear span of the set
\[
\{f_{F_1,F_2} : F_1\subset \mathcal{F}_0, F_2\subset D \text{ finite}\}
\]
where the coefficients come from the field $\Q[i]$. Then, it follows that $\mathcal{A}$ is an algebra, it is countable, it contains the constant function $1$, and it is closed under complex conjugation. Furthermore, $\mathcal{A}$ separates points of $Y$ since it contains the functions $\{ x\mapsto d(x,y) : y\in D\}$. Thus, $\mathcal{A}$ is dense in $C(Y)$ by the Stone-Weierstrass theorem. \\

Finally, if $f \in \mathcal{A}$ and $\rho = \rho_{(K',U')}$ for some $(K',U') \in \mathcal{F}_0$, it follows that $\rho f \in \mathcal{A}$. It is this property that is relevant to us in what follows. Note that, for a function $g\in C(Y)$, we write $\text{supp}(g)$ for the set $\{y\in Y : g(y) \neq 0\}$.

\begin{lemma}\label{lem: approximation_support}
Let $(Y,d)$ be a metric space and $\mathcal{A} \subset C(Y)$ be the subalgebra defined above. Let $f\in C(Y)$ be a non-negative function, and let $\delta > 0$. Then, there exists $f_0 \in \mathcal{A}$ such that $\|f-f_0\| < \delta$ and $\text{supp}(f_0) \subset \text{supp}(f)$.
\end{lemma}
\begin{proof}
Note that the algebra $\mathcal{A}$ is dense in $C(Y)$, so we may choose $f_1 \in \mathcal{A}$ such that $\|f-f_1\| < \delta/2$. Now set
\[
K := \{y\in Y : f(y) \geq \frac{\delta}{4}\} \text{ and } U := \text{supp}(f).
\]
Then $(K,U) \in \mathcal{F}$. By \cref{lem: order_dominate}, we may choose $(K',U') \in \mathcal{F}_0$ such that $(K,U) \leq (K',U')$. Let $\rho = \rho_{(K',U')}$ as above, and define
\[
f_0 := \rho f_1 \in \mathcal{A}.
\]
By construction, we have $\|\rho f - f\| \leq \frac{\delta}{2}$. Hence, $\|f_0 - f\| < \delta$. Finally, observe that $\text{supp}(f_0) \subset \text{supp}(\rho) \subset U' \subset U = \text{supp}(f)$.
\end{proof}

Finally, we arrive at the local characterization of finite Rokhlin dimension.

\begin{lemma}\label{lem: rokhlin_local}
Let $G$ be a compact, second countable group, and $A$ be a separable C*-algebra generated by a countable set $S = S^{\ast}$. For any action $\alpha : G\to \text{Aut}(A)$, one has $\dim_{Rok}(\alpha) \leq d$ if and only if, for each pair of finite sets $F\subset S, K\subset C(G),$ and each $\epsilon > 0$, there exists a  family of $(d+1)$ c.c.p. maps
\[
\psi_0, \psi_1,\ldots, \psi_d : C(G)\to A
\]
satisfying the following conditions:
\begin{enumerate}[label=(\roman*)]
\item For $f_1, f_2 \in K$ such that $f_1\perp f_2$, $\|\psi_j(f_1)\psi_j(f_2)\| < \epsilon$ for all $0\leq j\leq d$.
\item For any $a\in F$ and $f\in K$, we have $\|[\psi_j(f),a]\| < \epsilon$ for all $0\leq j\leq d$.
\item For any $f\in K$ and $s\in G$, we have $\|\alpha_s(\psi_j(f)) - \psi_j(\sigma_s(f))\| < \epsilon$.
\item For any $a\in F, \|\sum_{j=0}^d \psi_j(1_{C(G)})a - a\| < \epsilon$.
\end{enumerate}
\end{lemma}
A family of $(d+1)$ c.c.p. maps satisfying these conditions will henceforth be referred to as a \emph{$(d,F,K,\epsilon)$-Rokhlin system}.
\begin{proof}
We follow the proof of \cite[Lemma 3.7]{gardella_compact}. The `only if' direction of the proof is identical to that result, so we prove the `if' direction. Write
\[
S = \bigcup_{n=1}^{\infty} F_n
\]
where $\{F_n\}$ is an increasing collection of finite sets. Taking $Y = G$ in the above argument, define $\mathcal{A}$ as above. Choose an increasing sequence of finite sets $\{K_n\}$ such that
\[
\mathcal{A} = \bigcup_{n=1}^{\infty} K_n.
\]
For each $n\in \N$, let
\[
\psi_0^{(n)}, \psi_1^{(n)}, \ldots, \psi_d^{(n)} : C(G)\to A
\]
be a $(d,F_n,K_n,\frac{1}{n})$-Rokhlin system. Define $\phi_j : C(G)\to A_{\infty}$ by
\[
\phi_j(f) := \eta_A((\psi_j^{(n)}(f))).
\]
We claim that $\phi_j(f) \in A_{\infty,\alpha}$ for all $f\in C(G)$. Clearly, it suffices to show that $(\psi_j^{(n)}(f))_{n=1}^{\infty} \in \ell^{\infty}_{\alpha}(\N,A)$ for all $f\in \mathcal{A}$. So fix $f\in \mathcal{A}$ and $\epsilon > 0$, and choose $N \in \N$ such that $f\in K_N$ and $\frac{1}{N} < \frac{\epsilon}{3}$. For a fixed $s\in G$, there is an open set $U$ containing $s$ such that $\|\sigma_s(f) - \sigma_t(f)\| < \frac{\epsilon}{3}$ for all $t\in U$. By property (iii) of the hypothesis, and the fact that the $\psi_j^{(n)}$ are contractive, it follows that
\[
\|\alpha_s(\psi_j^{(n)}(f)) - \alpha_t(\psi_j^{(n)}(f))\| < \epsilon
\]
for all $n\geq N$ and all $t\in U$. Thus, $(\psi_j^{(n)}(f))_{n=1}^{\infty} \in \ell^{\infty}_{\alpha}(\N,A)$ as required. Now, we claim that $\phi_j(f) \in A'$ for all $f\in C(G)$. If $f\in \mathcal{A}$ and $a\in S$, then clearly $[\phi_j(f),a] = 0$. Since $S = S^{\ast}$, if $w$ is a word of elements from $S\cup S^{\ast}$, then $[\phi_j(f),w] = 0$. The same is true for linear combinations of such words, and by continuity, we conclude that $[\phi_j(f),a] = 0$ for all $f\in \mathcal{A}$ and all $a\in A$. By approximating in $C(G)$, we conclude that $\phi_j(f) \in A'$ for all $f\in C(G)$. Hence, we get a well-defined map
\[
\phi_j : C(G) \to A_{\infty,\alpha}\cap A'.
\]
Now, we wish to prove that the maps $\{\varphi_j := \kappa_A\circ \phi_j\}_{j=0}^d$ satisfy the requirements of \cref{defn: rokhlin_dim}.
\begin{enumerate}[label=(\roman*)]
\item Clearly, $\phi_j$ is c.c.p. since each $\psi_j^{(n)}$ is.
\item We prove that each $\phi_j$ has order zero. This is where, we believe, the proof of \cite[Lemma 3.7]{gardella_compact} has a small gap. Note that, to prove $\phi_j$ has order zero, it suffices to consider two non-negative functions $f,g\in C(Y)$ such that $f\perp g$, and then prove that $\phi_j(f)\perp \phi_j(g)$ (see \cite[Remark 2.4]{winter_zacharias}). \\

So fix non-negative functions $f,g\in C(G)$ such that $f\perp g$ and $\epsilon > 0$, and choose $\delta > 0$ such that for any $f_0, g_0 \in C(Y)$, if $\|\phi_j(g) - \phi_j(g_0)\| < \delta$, and $\|\phi_j(f) - \phi_j(f_0)\| < \delta$, then
\[
\|\phi_j(f)\phi_j(g) - \phi_j(f_0)\phi_j(g_0)\| < \epsilon
\]
(such a $\delta$ exists by the continuity of multiplication in $A_{\infty}$). Now, we apply \cref{lem: approximation_support}, to find $f_0, g_0 \in \mathcal{A}$ such that $\|f - f_0\| < \delta, \|g-g_0\| < \delta$, and $\text{supp}(f_0) \subset \text{supp}(f)$ and $\text{supp}(g_0) \subset \text{supp}(g)$. Hence, we have $f_0\perp g_0$. Now choose $N \in \N$ so that $f_0, g_0 \in K_N$, and $\frac{1}{N} < \epsilon$. Then, for any $n\geq N$, we have
\[
\|\psi_j^{(n)}(f_0)\psi_j^{(n)}(g_0)\| < \frac{1}{n} \leq \frac{1}{N} < \epsilon
\]
so that $\|\phi_j(f_0)\phi_j(g_0)\| < \epsilon$. Since $\phi_j$ are contractive, we have
\[
\|\phi_j(f) - \phi_j(f_0)\| < \delta \text{ and } \|\phi_j(g) - \phi_j(g_0)\| < \delta
\]
so that $\|\phi_j(f)\phi_j(g)\| < 2\epsilon$. This is true for any $\epsilon > 0$, so $\phi_j(f)\phi_j(g) = 0$. Since c.c.p. maps preserve adjoints, $\phi_j(f)$ and $\phi_j(g)$ are self-adjoint, so $\phi_j(g)\phi_j(f) = 0$ as well. Hence, $\phi_j(f) \perp \phi_j(g)$ as required.
\item Let $x := \sum_{j=0}^d \phi_j(1_{C(G)})$. Then, for any $a\in S$, we have $xa=a$. Since $S = S^{\ast}$, if $w$ is a word of elements in $S\cup S^{\ast}$, we conclude that $xw = w$. The same is true for any linear combination of such words. By continuity of left-multiplication by $x$, we conclude that $xa = a$ for all $a\in A$, so that $\kappa_A(x) = 1_{F_{\alpha}(A)}$.
\item Let $f\in C(G)$ and $h\in G$. We wish to verify that $(\alpha_{\infty})_h(\phi_j(f)) = \phi_j(\sigma_h(f))$. For a fixed $h\in G$, consider the map $C(G)\to A_{\infty}$ given by
\[
f \mapsto (\alpha_{\infty})_h(\phi_j(f)) - \phi_j(\sigma_h(f)).
\]
This map is continuous, so it suffices to verify the equality above for $f\in \mathcal{A}$. So fix $\epsilon > 0$ and choose $N \in \N$ such that $f\in K_N$ and $\frac{1}{N} < \epsilon$. Then, for all $n\geq N$, we have
\[
\|(\alpha_h(\psi_j^{(n)}(f)) - \psi_j^{(n)}(\sigma_h(f))\| < \frac{1}{n} \leq \frac{1}{N} < \epsilon.
\]
Hence, $\|(\alpha_{\infty})_h(\phi_j(f)) - \phi_j(\sigma_h(f))\| < \epsilon$. This is true for any $\epsilon > 0$, so we obtain the desired equality.
\end{enumerate}
Hence, we conclude that the maps $\varphi_j := \kappa_A\circ \phi_j, 0\leq j\leq d$ satisfy the requirements of \cref{defn: rokhlin_dim}.
\end{proof}

\section{$C(X)$-algebras}\label{sec: cx_algebras}

By a theorem of Vasselli (see \cref{thm: vasselli} below), the Cuntz-Pimsner algebra associated to a vector bundle over a compact Hausdorff space $X$ is isomorphic to a $C(X)$-algebra. In order to use this profitably, we show that, under certain conditions, finiteness of Rokhlin dimension passes from the fibers of a $C(X)$-algebra to the ambient algebra. \\

We begin with some definitions: Let $X$ be a compact, Hausdorff space. A $C(X)$-algebra is a C*-algebra $A$ together with a unital $\ast$-homomorphism $\Theta:C(X)\to \mathcal{Z}(M(A))$, where $\mathcal{Z}(M(A))$ denotes the center of the multiplier algebra of $A$. To avoid notational complexity, for a function $f\in C(X)$ and $a\in A$, we simply write $fa := \Theta(f)(a)$. Let $Y\subset X$ be a closed set, and let $C(X,Y)$ denote the ideal of functions in $C(X)$ that vanish on $Y$. Then, $C(X,Y)A$ is a closed ideal of $A$ (by the Cohen factorization theorem \cite[Theorem 4.6.4]{brown_ozawa}), so we write $A(Y) := A/C(X,Y)A$ for the corresponding quotient, and $\pi_Y : A\to A(Y)$ for the quotient map. Furthermore, if $Y = \{x\}$ is a singleton set, then $A(x) := A(\{x\})$ is called the fiber of $A$ at $x$, and we write $\pi_x : A\to A(x)$ for the corresponding quotient map. For $a\in A$, we simply write $a(x)$ for $\pi_x(a) \in A(x)$. For each $a\in A$, we have a map $\Gamma_a : X\to \R$ given by $x\mapsto \|a(x)\|$. An important fact, and one that we will use repeatedly below, is that this function is upper semi-continuous (see \cite[Proposition 1.2]{rieffel}). We say that $A$ is a \emph{continuous} $C(X)$-algebra if each $\Gamma_a$ is continuous. \\

Given a $C(X)$-algebra $A$, an automorphism $\beta \in \text{Aut}(A)$ is said to be $C(X)$-linear if $\beta(fa) = f\beta(a)$ for all $a\in A$ and $f\in C(X)$. We write $\text{Aut}_X(A)$ for the collection of $C(X)$-linear automorphisms of $A$. Note that $\text{Aut}_X(A)$ is a subgroup of $\text{Aut}(A)$. Given an automorphism $\beta \in \text{Aut}_X(A)$, and a closed set $Y\subset X$, there is a natural automorphism $\beta_Y \in \text{Aut}(A(Y))$ such that $\beta_Y\circ \pi_Y = \pi_Y\circ \beta$. \\

In particular, if $\alpha : G\to \text{Aut}_X(A)$ is an action of a group $G$ on $A$ by $C(X)$-linear automorphisms, then, for each closed set $Y\subset X$, there is an action $\alpha_Y : G\to \text{Aut}(A(Y))$ such that the quotient map $\pi_Y:A\to A(Y)$ is $G$-equivariant. As before, if $Y = \{x\}$ is a singleton set, we write $\alpha_x$ for $\alpha_{\{x\}}$. The next lemma is due to Gardella and Lupini \cite[Proposition 4.26]{gardella_lupini}, and is stated below in the form that we need. 

\begin{lemma}\label{lem: equivariant_section}
Let $A$ be a separable, nuclear $C(X)$-algebra and $\alpha : G\to \text{Aut}_X(A)$ be an action of a compact, second countable group $G$ on $A$ by $C(X)$-linear automorphisms. Let $Y\subset X$ be a closed set, and let $\pi_Y : A \to A(Y)$ be the quotient map. Then there is a $G$-equivariant c.c.p. section $s : A(Y)\to A$ such that $\pi_Y\circ s = \text{id}_{A(Y)}$.
\end{lemma}
\begin{proof}
That there is a c.c.p. section $\widetilde{s} : A(Y)\to A$ follows from the Choi-Effros theorem (see \cite[Theorem C.3]{brown_ozawa}). Now we average over $G$ by defining
\[
s(b) := \int_G \alpha_g(\widetilde{s}((\alpha_Y)_{g^{-1}}(b)))d\mu(g)
\]
where $\mu$ is the normalized Haar measure on $G$. It is easy to see that this map satisfies the required properties.
\end{proof}

The next lemma is a strengthening of \cite[Proposition 1.6]{kirchberg_winter} that we need for our purposes.

\begin{lemma}\label{lem: strongly_decomposable_cover}
Let $X$ be a compact Hausdorff space with $\dim(X) \leq n$. Then, every finite open cover has a finite refinement $\mathcal{U}$ which may be decomposed as the disjoint union of subcollections
\[
\mathcal{U} = \mathcal{U}_0\sqcup\mathcal{U}_1\sqcup \ldots \sqcup \mathcal{U}_n
\]
such that, for each $0\leq i\leq n$,
\[
V_1,V_2 \in \mathcal{U}_i, V_1\neq V_2 \Rightarrow \overline{V_1}\cap \overline{V_2} = \emptyset.
\]
\end{lemma}
We call such a finite refinement a \emph{strongly $n$-decomposable} cover.
\begin{proof}
In light of \cite[Proposition 1.6]{kirchberg_winter}, it suffices to show that, for any finite open cover $\mathcal{U} = \{U_1,U_2,\ldots, U_n\}$ of $X$, there is another open cover $\mathcal{V} = \{V_1,V_2,\ldots, V_n\}$ of $X$ such that $\overline{V_i} \subset U_i$ for all $1\leq i\leq n$. \\

So fix a finite open cover $\mathcal{U} = \{U_1, U_2,\ldots, U_n\}$, and assume, without loss of generality, that $n\geq 2$. We now construct $V_k$ for $k\geq 1$ inductively. Fix $1\leq k<n$, and suppose that we have constructed $\{V_i : i<k\}$ such that $\overline{V_i}\subset U_i$ for all $i<k$, and
\[
X = \left(\bigcup_{i<k} V_i\right) \cup \left(\bigcup_{j\geq k} U_j\right)
\]
with the understanding that the first union is the empty set if $k=1$. To construct $V_k$, we note that, if 
\[
W := \left(\bigcup_{i<k} V_i\right) \cup \left(\bigcup_{j\geq k+1} U_j\right)
\]
then $X\setminus W\subset U_k$. Since $X\setminus W$ is compact and $X$ is normal, there is an open set $V_k$ such that $X\setminus W\subset V_k \subset \overline{V_k}\subset U_k$. Then, observe that
\[
X = \left(\bigcup_{i\leq k} V_i\right) \cup \left(\bigcup_{j\geq k+1} U_j\right)
\]
so we may proceed in this way to construct the required cover $\mathcal{V}$.
\end{proof}

We now prove the main result of this section.

\begin{theorem}\label{thm: rokhlin_cx_algebra}
Let $X$ be a compact, Hausdorff space, $A$ be a separable, nuclear $C(X)$-algebra, and let $\alpha: G\to \text{Aut}_X(A)$ be an action of a compact, second countable group $G$ on $A$, acting by $C(X)$-linear automorphisms. Then
\[
\dim_{Rok}(\alpha) \leq (\dim(X) + 1)\left[(\sup_{x\in X} \dim_{Rok}(\alpha_x)) + 1\right] -1.
\]
\end{theorem}
\begin{proof}
Assume, without loss of generality, that
\[
n := \dim(X) < \infty \text{ and } d := \sup_{x\in X}\dim_{Rok}(\alpha_x) < \infty.
\]
Our goal is to verify \cref{lem: rokhlin_local}. Fix finite sets $F\subset A, K\subset C(G)$ and $\eta > 0$. Set
\[
\epsilon := \frac{\eta}{(n+1)}.
\]
For $x\in X$ fixed, then there are $(d+1)$ c.c.p. maps $\psi_0, \psi_1, \ldots, \psi_d : C(G)\to A(x)$ which form a $(d,\pi_x(F), K, \epsilon/3)$-Rokhlin system. Since $G$ is second countable, $C(G)$ is separable. Therefore, by the Choi-Effros theorem \cite[Theorem C.3]{brown_ozawa}, there are c.c.p. maps
\[
\widetilde{\psi_k} : C(G)\to A,\quad 0\leq k\leq d
\]
such that $\pi_x\circ \widetilde{\psi_k} = \psi_k$ for all $0\leq k\leq d$.
\begin{enumerate}[label=(\roman*)]
\item Now note that, if $f_1,f_2\in K$ with $f_1\perp f_2$, then
\[
\|\pi_x(\widetilde{\psi_k}(f_1))\pi_x(\widetilde{\psi_k}(f_2))\| < \frac{\epsilon}{3}.
\]
For $b := \widetilde{\psi_k}(f_1)\widetilde{\psi_k}(f_2)$, the map $\Gamma_b$ defined above is upper semi-continuous, so there is an open set $W_1$ containing $x$ such that
\[
\|\pi_y(\widetilde{\psi_k}(f_1))\pi_y(\widetilde{\psi_k})(f_2)\| < \epsilon
\]
for all $y\in W_1$. Choose an open set $V_1$ containing $x$ such that $\overline{V_1} \subset W_1$. Then, by \cite[Lemma 2.1]{mdd_finite},
\[
\|\pi_{\overline{V_1}}(\widetilde{\psi_k}(f_1))\pi_{\overline{V_1}}(\widetilde{\psi_k}(f_2))\| < \epsilon.
\]
Since $K$ is finite, we may arrange that this happens for all pairs $f_1,f_2 \in K$ such that $f_1\perp f_2$.
\item Similarly, since
\[
\|\left(\sum_{k=0}^d \pi_x(\widetilde{\psi_k}(1_{C(G)}))\right)a(x) - a(x)\| < \frac{\epsilon}{3}
\]
for all $a\in F$, by upper semi-continuity of the maps $\{\Gamma_b : b\in A\}$, there is an open set $V_2$ containing $x$ such that
\[
\|\left(\sum_{k=0}^d \pi_{\overline{V_2}}(\widetilde{\psi_k}(1_{C(G)}))\right)\pi_{\overline{V_2}}(a) - \pi_{\overline{V_2}}(a)\| < \epsilon
\]
for all $a\in F$. 
\item For a fixed $0\leq k\leq d$ and $f\in F$, the maps $h\mapsto \alpha_h\circ \widetilde{\psi}_k(f)$ and $h\mapsto \sigma_h(f)$ are continuous maps from $G$ to $A$ and $G$ to $C(G)$ respectively. Hence, for each $s\in G$, there is an open set $U_s$ containing $s$ such that $\|\alpha_s(\widetilde{\psi}_k(f)) - \alpha_t(\widetilde{\psi}_k(f))\| < \frac{\epsilon}{3}$, and $\|\sigma_s(f) - \sigma_t(f)\|_{\infty} < \frac{\epsilon}{3}$ for all $t\in U_s$ and $0\leq k\leq d, f\in F$. By compactness, we may choose a finite set $M \subset G$ such that $G = \bigcup_{s\in M} U_s$. Now, for each $s\in M$, we have 
\[
\|(\alpha_x)_s(\pi_x(\widetilde{\psi_k}(f))) - \pi_x(\widetilde{\psi_k}(\sigma_s(f)))\| < \frac{\epsilon}{3}
\]
for all $0\leq k\leq d$ and $f\in F$. Since $\alpha_s$ is $C(X)$-linear, we have $(\alpha_x)_s \circ \pi_x = \pi_x\circ \alpha_s$. So, we may once again use the upper semi-continuity of the maps $\{\Gamma_b : b\in A\}$ to obtain an open set $W_s$ containing $x$ such that
\[
\|\pi_{\overline{W_s}}\left( \alpha_s(\widetilde{\psi_k}(f)) - \widetilde{\psi_k}(\sigma_s(f))\right)\| < \frac{\epsilon}{3}.
\]
Take $V_3 := \bigcap_{s\in M} W_s$, then $V_3$ is open, and
\[
\|\pi_{\overline{V_3}}\left( \alpha_s(\widetilde{\psi_k}(f)) - \widetilde{\psi_k}(\sigma_s(f))\right)\| < \frac{\epsilon}{3}
\]
for each $s\in M$. Now, for any $t\in G$, choose $s\in M$ such that $t \in U_s$. Then, for any $f\in F$, we have
\begin{equation*}
\begin{split}
\|(\alpha_{\overline{V_3}})_t(\pi_{\overline{V_3}}(\widetilde{\psi_k}(f))) - \pi_{\overline{V_3}}(\widetilde{\psi_k}(\sigma_t(f)))\| < \epsilon
\end{split}
\end{equation*}
by the triangle inequality, the fact that $(\alpha_{\overline{V_3}})_t\circ \pi_{\overline{V_3}} = \pi_{\overline{V_3}}\circ \alpha_t$, and that $\pi_{\overline{V_3}}\circ \widetilde{\psi_k}$ is contractive.
\item Finally, using upper semi-continuity as before, there is an open set $V_4$ such that
\[
\|[\pi_{\overline{V_4}}(\widetilde{\psi_k}(f)), \pi_{\overline{V_4}}(a)]\| < \epsilon 
\]
for all $a\in F, f\in K$ and $0\leq k\leq d$.
\end{enumerate}
Therefore, if we set $V_x := V_1\cap V_2\cap V_3\cap V_4$, then we see that, for each $x\in X$, there is an open set $V_x$ containing $x$ and $(d+1)$ c.c.p. maps
\[
\psi^{(k)}_x = \pi_{\overline{V_x}}\circ \widetilde{\psi_k} : C(G)\to A(\overline{V_x})
\]
for $0\leq k\leq d$, which form a $(d, \pi_{\overline{V_x}}(F), K, \epsilon)$-Rokhlin system. \\

Now, begin with a finite subcover of $\{V_x : x\in X\}$, and choose a strongly $n$-decomposable refinement of this cover (by \cref{lem: strongly_decomposable_cover}), written as $\mathcal{U} = \mathcal{U}_0\sqcup \mathcal{U}_1 \sqcup \ldots \sqcup \mathcal{U}_n$, where $\mathcal{U}_i = \{V_{i,1}, V_{i,2}, \ldots, V_{i,k_i}\}$, and $\overline{V_{i,j_1}} \cap \overline{V_{i,j_2}} = \emptyset$ if $j_1\neq j_2$. For each $0\leq i\leq n$ and $1\leq j\leq k_i$, there are $(d+1)$ c.c.p. maps
\[
\psi^{(k)}_{i,j} : C(G)\to A(\overline{V_{i,j}}), \quad 0\leq k\leq d
\]
which form a $(d,\pi_{\overline{V_{i,j}}}(F), K,\epsilon)$-Rokhlin system. Write $V_i = \sqcup_{j=1}^{k_i} V_{i,j}$, so that $\overline{V_i} = \sqcup_{j=1}^{k_i} \overline{V_{i,j}}$. By \cite[Lemma 2.4]{mdd_finite},
\[
A(\overline{V_i}) \cong \bigoplus_{j=1}^{k_i} A(\overline{V_{i,j}}).
\]
So, we get $(d+1)$ c.c.p. maps
\[
\psi_i^{(k)} : C(G)\to A(\overline{V_i}), \qquad 0\leq k\leq d
\]
which form a $(d,\pi_{\overline{V_i}}(F), K,\epsilon)$-Rokhlin system. \\

Now, choose a partition of unity $\{f_i : 0\leq i\leq n\}$ subordinate to this open cover $\{V_i: 0\leq i\leq n\}$ and $G$-equivariant, c.c.p. sections $s_i : A(\overline{V_i}) \to A$ satisfying $\pi_{\overline{V_i}}\circ s_i = \text{id}_{A(\overline{V_i})}$ using \cref{lem: equivariant_section}. Define $\varphi_{i,k} : C(G)\to A$ by
\[
f \mapsto f_is_i(\psi_i^{(k)}(f)).
\]
Then each $\varphi_{i,k}$ is an c.c.p. map. We claim that the $\{\varphi_{i,k}\}$ form a $((d+1)(n+1)-1, F,K,\eta)$-Rokhlin system.
\begin{enumerate}[label=(\roman*)]
\item If $f_1,f_2\in K$ with $f_1\perp f_2$. If $b_1 := \psi_i^{(k)}(f_1)$ and $b_2 := \psi_i^{(k)}(f_2)$, then $\|b_1b_2\| < \epsilon$. Furthermore,
\[
\pi_{\overline{V_i}}(s_i(b_1b_2) - s_i(b_1)s_i(b_2)) = 0.
\]
Hence, $s_i(b_1b_2) - s_i(b_1)s_i(b_2) \in C(X,\overline{V_i})A$. Since $f_i \in C_0(V_i)$, we have $f_is_i(b_1b_2) = f_is_i(b_1)s_i(b_2)$. Hence,
\begin{equation*}
\begin{split}
\|\varphi_{i,k}(f_1)\varphi_{i,k}(f_2)\| &= \|f_i^2s_i(b_1)s_i(b_2)\| = \|f_i^2s_i(b_1b_2)\| \leq \|b_1b_2\| < \epsilon.
\end{split}
\end{equation*}
\item We wish to show that
\[
\|\left(\sum_{i=0}^n \sum_{k=0}^d \varphi_{i,k}(1_{C(G)})\right)a - a\| < \eta
\]
for all $a\in F$. Note that $\pi_{\overline{V_i}}(a - s_i(\pi_{\overline{V_i}}(a))) = 0$, so $f_ia = f_is_i(\pi_{\overline{V_i}}(a))$. Hence,
\begin{equation*}
\begin{split}
\left\|\sum_{i,k} \varphi_{i,k}(1_{C(G)})a - a\right\| &= \left\|\left(\sum_{i=0}^n\sum_{k=0}^d f_is_i(\psi_i^{(k)}(1_{C(G)}))\right)a -a\right\|\\
&= \left\|\left(\sum_{i=0}^n f_i\sum_{k=0}^d s_i(\psi_i^{(k)}(1_{C(G)}))s_i(\pi_{\overline{V_i}}(a))\right) - \left(\sum_{i=0}^n f_i\right)a\right\| \\
&= \left\|\sum_{i=0}^n f_i\left(\sum_{k=0}^d s_i(\psi_i^{(k)}(1_{C(G)}))s_i(\pi_{\overline{V_i}}(a)) - s_i(\pi_{\overline{V_i}}(a)) \right)\right\|.
\end{split}
\end{equation*}
As argued above, we have
\[
f_is_i(\psi_i^{(k)}(1_{C(G)}))s_i(\pi_{\overline{V_i}}(a)) = f_is_i(\psi_i^{(k)}(1_{C(G)})\pi_{\overline{V_i}}(a)).
\]
Hence,
\begin{equation*}
\begin{split}
\left\|\sum_{i,k} \varphi_{i,k}(1_{C(G)})a - a\right\| &= \left\| \sum_{i=0}^n f_i\left(\sum_{k=0}^d s_i(\psi_i^{(k)}(1_{C(G)})\pi_{\overline{V_i}}(a)) - s_i(\pi_{\overline{V_i}}(a)) \right)\right\|\\
&\leq (n+1)\left\|\sum_{k=0}^d \psi_i^{(k)}(1_{C(G)})\pi_{\overline{V_i}}(a) - \pi_{\overline{V_i}}(a)\right\| \\
&< (n+1)\epsilon = \eta.
\end{split}
\end{equation*}
\item We need to verify that $\|[\varphi_{i,k}(f), a]\| < \eta$ for all $f\in K, a\in F$. Once again, if $b_1 := \psi_i^{(k)}(f)$ and $b_2 := \pi_{\overline{V_i}}(a)$, then $\|[b_1,b_2]\| < \epsilon$. Furthermore, $f_ia = f_is_i(\pi_{\overline{V_i}}(a)) = f_is_i(b_2)$, and $f_is_i(b_1b_2) = f_is_i(b_1)s_i(b_2)$. Hence,
\begin{equation*}
\begin{split}
\|[\varphi_{i,k}(f), a]\| &= \left\|f_is_i(\psi_i^{(k)}(f))a - f_ias_i(\psi_i^{(k)}(f)) \right\| \\
&= \left\|f_is_i(b_1)s_i(b_2) - f_is_i(b_2)s_i(b_1) \right\| \\
&= \|f_is_i(b_1b_2) - f_is_i(b_2b_1)\| \\
&\leq \|b_1b_2-b_2b_1\| < \epsilon < \eta.
\end{split}
\end{equation*}
\item We need to verify that $\|\alpha_h(\varphi_{i,k}(f)) - \varphi_{i,k}(\sigma_h(f))\| < \eta$ for all $h\in G, f\in K$. But note that
\[
\alpha_h\circ s_i\circ \psi_i^{(k)}(f) = s_i\circ (\alpha_{\overline{V_i}})_h\circ \psi_i^{(k)}(f).
\]
Hence,
\begin{equation*}
\begin{split}
\|\alpha_h(\varphi_{i,k}(f)) - \varphi_{i,k}(\sigma_h(f))\| &= \left\|f_i\alpha_h\circ s_i\circ \psi_i^{(k)}(f) -  f_is_i\circ \psi_i^{(k)}(\sigma_h(f))\right\| \\
&= \left\|f_i\left(s_i\circ (\alpha_{\overline{V_i}})_h\circ\psi_i^{(k)}(f) - s_i\circ \psi_i^{(k)}(\sigma_h(f))\right)\right\| \\
&= \left\|f_i\left(s_i((\alpha_{\overline{V_i}})_h(\psi_i^{(k)}(f)))) - s_i(\psi_i^{(k)}(\sigma_h(f)))\right) \right\| \\
&\leq \|(\alpha_{\overline{V_i}})_h\circ\psi_i^{(k)}(f) - \psi_i^{(k)}(\sigma_h(f))\| < \epsilon \\
&<\eta.
\end{split}
\end{equation*}
\end{enumerate}
Thus, we have constructed $(n+1)(d+1)$ c.c.p. maps $\varphi_{i,k} : C(G)\to A$ forming a $((d+1)(n+1)-1,F,K,\eta)$-Rokhlin system. Since $A$ is separable, we conclude by \cref{lem: rokhlin_local} that
\[
\dim_{Rok}(\alpha) \leq (n+1)(d+1) - 1
\]
as required.
\end{proof}

\section{C*-Correspondences}\label{sec: correspondences}

In this section, we study a natural situation in which an action of a group on a C*-correspondence may induce an action of finite Rokhlin dimension on the associated Cuntz-Pimsner algebra. To begin with, we briefly review the basic notion of a C*-correspondence and the associated Pimsner algebras. For further details, the reader is referred to \cite[Section 4.6]{brown_ozawa} or \cite{katsura}. \\

Let $A$ be a C*-algebra. A C*-correspondence over $A$ is a triple $(\mathcal{H},A,\pi_\mathcal{H})$, where $\mathcal{H}$ is a right Hilbert $A$-module, and $\pi_\mathcal{H} : A\to \mathcal{B}(\mathcal{H})$ is a $\ast$-representation of $A$ on the algebra of adjointable operators on $\mathcal{H}$. For $\xi, \eta \in \mathcal{H}$, we write $\theta_{\xi,\eta} \in \mathcal{B}(\mathcal{H})$ for the operator $\zeta \mapsto \xi\langle \eta,\zeta\rangle$, and write $\mathcal{K}(\mathcal{H})$ for closed linear span of $\{\theta_{\xi,\eta} : \xi, \eta \in \mathcal{H}\}$. This is a closed two-sided ideal of $\mathcal{B}(\mathcal{H})$. \\

A \emph{representation} of the triple $(\mathcal{H},A,\pi_\mathcal{H})$ on a C*-algebra $B$ is a pair $(\pi, \tau)$, where $\pi : A\to B$ is a $\ast$-homomorphism, and $\tau : \mathcal{H}\to B$ is a linear map such that
\[
\tau(\pi_\mathcal{H}(a)\xi) = \pi(a)\tau(\xi), \text{ and } \tau(\xi)^{\ast}\tau(\eta) = \pi(\langle \xi,\eta)\rangle).
\]
Write $C^{\ast}(\pi,\tau)$ for the C*-subalgebra of $B$ generated by $\pi(A)\cup \tau(\mathcal{H})$. A representation $(\widetilde{\pi}, \widetilde{\tau})$ is said to be \emph{universal} if, for any other representation $(\pi,\tau)$ of $(\mathcal{H},A,\pi_\mathcal{H})$, there is a $\ast$-homomorphism $\mu: C^{\ast}(\widetilde{\pi},\widetilde{\tau})\to C^{\ast}(\pi,\tau)$ such that
\[
\mu\circ \widetilde{\pi} = \pi \text{ and } \mu\circ \widetilde{\tau} = \tau.
\]
It is a fact that such a universal representation always exists. The C*-algebra $C^{\ast}(\widetilde{\pi}, \widetilde{\tau})$ is called is the (augmented) Toeplitz-Pimsner algebra associated to $(\mathcal{H},A,\pi_\mathcal{H})$, and is denoted by $\mathcal{T}_\mathcal{H}$. \\

Given a representation $(\pi,\tau)$ of $(\mathcal{H},A,\pi_\mathcal{H})$ on a C*-algebra $B$, there is a unique $\ast$-homomorphism $\sigma_{\tau} : \mathcal{K}(\mathcal{H})\to B$ satisfying $\sigma_{\tau}(\theta_{\xi,\eta}) = \tau(\xi)\tau(\eta)^{\ast}$. We say that the representation is \emph{covariant} if
\[
\pi(a) = \sigma_{\tau}(\pi_\mathcal{H}(a))
\]
for all $a\in J_\mathcal{H}$, where $J_\mathcal{H}$ is defined by
\[
J_\mathcal{H} := \{a\in A:\pi_\mathcal{H}(a) \in \mathcal{K}(\mathcal{H}) \text{ and } ab = 0 \text{ for all } b\in \ker(\pi_\mathcal{H})\}.
\]
A covariant representation $(\widehat{\pi}, \widehat{\tau})$ is said to be \emph{universal} if, for any other covariant representation $(\pi,\tau)$ of $(\mathcal{H},A,\pi_\mathcal{H})$, there is a $\ast$-homomorphism $\nu : C^{\ast}(\widehat{\pi},\widehat{\tau}) \to C^{\ast}(\pi,\tau)$ satisfying
\[
\nu\circ \widehat{\pi} = \pi \text{ and } \nu\circ \widehat{\tau} = \tau.
\]
Once again, such a universal covariant representation always exists. The C*-algebra $C^{\ast}(\widehat{\pi},\widehat{\tau})$ is called the (augmented) Cuntz-Pimsner algebra associated to $(\mathcal{H},A,\pi_\mathcal{H})$, and is denoted by $\mathcal{O}_\mathcal{H}$.\\

We now recall the notion of a group action on a C*-correspondence (see, for instance, \cite[Definition 2.1]{hao_ng}).

\begin{definition}
A continuous action of a locally compact group $G$ on a C*-correspondence $(\mathcal{H},A,\pi_\mathcal{H})$ is a pair $(\alpha,\gamma)$, where $(A,G,\alpha)$ is a C*-dynamical system, and $\gamma : G\to \text{Aut}(\mathcal{H})$ is a group homomorphism such that, for any $s\in G, \xi, \eta \in \mathcal{H}$ and $a\in A$, 
\begin{enumerate}[label=(\roman*)]
\item The map $t\mapsto \gamma_t(\xi)$ is a continuous map from $G$ to $\mathcal{H}$,
\item $\langle \gamma_s(\xi),\gamma_s(\eta)\rangle = \alpha_s(\langle \xi,\eta\rangle)$,
\item $\gamma_s(\xi\cdot a) = \gamma_s(\xi)\cdot \alpha_s(a)$,
\item $\gamma_s(\pi_\mathcal{H}(a)\xi) = \pi_\mathcal{H}(\alpha_s(a))\gamma_s(\xi)$.
\end{enumerate}
\end{definition}

The proof of the next lemma follows directly from the universal properties of the Pimsner algebras mentioned above, and its proof has been discussed elsewhere (see \cite[Corollary 4.6.22]{brown_ozawa} and \cite[Lemma 2.6]{hao_ng}). We omit the details here.

\begin{lemma}\label{lem: action_pimsner}
Let $(\alpha, \gamma)$ be an action of a locally compact group $G$ on a C*-correspondence $(\mathcal{H},A,\pi_\mathcal{H})$. Then,
\begin{enumerate}[label=(\roman*)]
\item There is a unique action $\mu : G\to \text{Aut}(\mathcal{T}_\mathcal{H})$ such that
\[
\mu_s\circ \widetilde{\pi} = \widetilde{\pi}\circ \alpha_s \text{ and } \mu_s\circ \widetilde{\tau} = \widetilde{\tau} \circ \gamma_s
\]
for all $s\in G$.
\item There is a unique action $\beta : G\to \text{Aut}(\mathcal{O}_\mathcal{H})$ such that
\[
\beta_s\circ \widehat{\pi} = \widehat{\pi}\circ \alpha_s \text{ and } \beta_s\circ \widehat{\tau} = \widehat{\tau}\circ \gamma_s
\]
for all $s\in G$.
\end{enumerate}
\end{lemma}

The next definition is meant to ensure that an action on a C*-correspondence induces an action of finite Rokhlin dimension on the associated Pimsner algebras.

\begin{definition}\label{definition: rokhlin_correspondence_compact}
Let $(\alpha, \gamma)$ be an action of a compact group $G$ on a C*-correspondence $(\mathcal{H},A, \pi_\mathcal{H})$, and let $d\geq 0$ be a fixed integer. We say that $(\alpha,\gamma)$ is \emph{$d$-decomposable} if, for every finite set $F_1\subset A, F_2\subset \mathcal{H}, K\subset C(G),$ and every $\epsilon > 0$, there exist $(d+1)$ c.c.p. maps
\[
\varphi_0, \varphi_1, \ldots, \varphi_d : C(G)\to A
\]
satisfying the following conditions:
\begin{enumerate}[label=(\roman*)]
\item If $f_1,f_2\in K$ are orthogonal, then $\|\varphi_j(f_2)\varphi_j(f_2)\| < \epsilon$.
\item If $a \in F_1$, then $\|\left(\sum_{j=0}^d \varphi_j(1_{C(G)})\right)a - a\| < \epsilon$.
\item If $\xi \in F_2$, then
\begin{equation*}
\begin{split}
\|\pi_\mathcal{H}\left(\sum_{j=0}^d \varphi_j(1_{C(G)})\right)\xi - \xi\| &< \epsilon, \text{ and} \\
\|\xi\left(\sum_{j=0}^d \varphi_j(1_{C(G)})\right) -\xi\| &< \epsilon.
\end{split}
\end{equation*}
\item For all $a\in F_1$, and all $f\in K$, $\|[\varphi_j(f),a]\| < \epsilon$.
\item For all $\xi \in F_2$, and $f\in K$, $\|\pi_\mathcal{H}(\varphi_j(f))\xi - \xi\cdot \varphi_j(f)\| < \epsilon$.
\item For all $f\in K$ and $h\in G$, $\|\alpha_h(\varphi_j(f)) - \varphi_j(\sigma_h(f))\| < \epsilon$.
\end{enumerate}
\end{definition}

We now show that a $d$-decomposable action induces an action of finite Rokhlin dimension on the associated Pimsner algebras. Recall that, if $(\pi,\tau)$ is a representation of $(\mathcal{H},A,\pi_\mathcal{H})$, then $\tau(\xi)\pi(a) = \tau(\xi\cdot a)$, and $\|\tau(\xi)\| \leq \|\xi\|$ for all $\xi \in \mathcal{H}$ and $a\in A$. We use these facts below.

\begin{prop}\label{prop: rokhlin_pimsner}
Let $(\mathcal{H},A,\pi_\mathcal{H})$ be a Hilbert C*-correspondence, where $A$ is a separable C*-algebra and $\mathcal{H}$ a countably generated Hilbert $A$-module. Let $(\alpha,\gamma)$ be an action of a compact group $G$ on $(\mathcal{H},A,\pi_\mathcal{H})$, and let $\mu: G\to \text{Aut}(\mathcal{T}_\mathcal{H})$ and $\beta : G\to \text{Aut}(\mathcal{O}_\mathcal{H})$ denote the induced actions on Pimsner algebras given by \cref{lem: action_pimsner}. If $(\alpha,\gamma)$ is $d$-decomposable, then
\[
\dim_{Rok}(\mu) \leq d \text{ and } \dim_{Rok}(\beta) \leq d.
\]
\end{prop}
\begin{proof}
We prove that $\dim_{Rok}(\beta) \leq d$ since the other case is similar. Since $A$ is separable and $\mathcal{H}$ is countably generated, $\mathcal{O}_\mathcal{H}$ is separable, so it suffices to verify the conditions of \cref{lem: rokhlin_local}. So, we choose countable sets $T_1 \subset A$ and $T_2\subset \mathcal{H}$ which each generate $A$ and $\mathcal{H}$ respectively with $T_1 = T_1^{\ast}$. Set
\[
S_1 := \{\widehat{\pi}(a) : a\in T_1\}, \text{ and } S_2 := \{\widehat{\tau}(\xi), \widehat{\tau}(\xi)^{\ast} : \xi \in T_2\}.
\]
Then, if $S := S_1\cup S_2$, then $S$ is a countable set which generates $\mathcal{O}_\mathcal{H}$ satisfying $S = S^{\ast}$. Therefore, we are in a position to apply \cref{lem: rokhlin_local}. So fix finite sets $F \subset S, K\subset C(G),$ and $\epsilon > 0$, and assume without loss of generality that $K=K^{\ast}$. Set
\begin{equation*}
\begin{split}
F_1 &:= \{a \in T_1 : \widehat{\pi}(a) \in F\}, \text{ and } \\
F_2 &:= \{\xi \in T_2 : \text{ either } \widehat{\tau}(\xi) \in F, \text{ or } \widehat{\tau}(\xi)^{\ast} \in F\}
\end{split}
\end{equation*}
By hypothesis, there are $(d+1)$ c.c.p. maps
\[
\varphi_0, \varphi_1, \ldots, \varphi_d : C(G)\to A
\]
satisfying the conditions of \cref{definition: rokhlin_correspondence_compact} for the tuple $(F_1,F_2,K,\epsilon)$. Define $\psi_j : C(G)\to \mathcal{O}_\mathcal{H}$ by
\[
\psi_j := \widehat{\pi}\circ \varphi_j.
\]
Then each $\psi_j : C(G)\to \mathcal{O}_\mathcal{H}$ is a c.c.p. map.
\begin{enumerate}[label=(\roman*)]
\item Clearly, if $f_1,f_2\in K$ are orthogonal, then $\|\psi_j(f_1)\psi_j(f_2)\| < \epsilon$.
\item If $x\in F$, we wish to prove that $\|(\sum_{j=0}^d \psi_j(1_{C(G)})x - x\| < \epsilon$. We consider the following cases separately:
\begin{itemize}
\item If $x = \widehat{\pi}(a)$ for some $a\in F_1$, then
\[
\|\sum_{j=0}^d \psi_j(1_{C(G)})x - x\| \leq \|\sum_{j=0}^d \varphi_j(1_{C(G)})a - a\| < \epsilon.
\]
\item If $x = \widehat{\tau}(\xi)$ for some $\xi \in F_2$. Then
\begin{equation*}
\begin{split}
\|\sum_{j=0}^d \psi_j(1_{C(G)})x - x\| &= \|\sum_{j=0}^d \widehat{\pi}(\varphi_j(1_{C(G)}))\widehat{\tau}(\xi) - \widehat{\tau}(\xi)\| \\
&= \left\|\widehat{\tau}\left(\sum_{j=0}^d \pi_\mathcal{H}(\varphi_j(1_{C(G)}))\xi - \xi\right)\right\| \\
&< \epsilon.
\end{split}
\end{equation*}
\item If $x = \widehat{\tau}(\xi)^{\ast}$ for some $\xi \in F_2$, then the argument is similar after taking adjoints.
\end{itemize}
\item If $x\in F$, and $f\in K$, then we wish to prove that $\|[\psi_j(f), x]\| < \epsilon$. Once again, we consider the following cases separately:
\begin{itemize}
\item If $x = \widehat{\pi}(a)$ for some $a\in F_1$, then
\begin{equation*}
\begin{split}
\|[\psi_j(f), x]\| &= \|\widehat{\pi}([\varphi_j(f),a])\| \leq \|[\varphi_j(f),a]\| < \epsilon.
\end{split}
\end{equation*}
\item If $x = \widehat{\tau}(\xi)$ for some $\xi \in F_2$, then
\begin{equation*}
\begin{split}
\|[\psi_j(f),x]\| &= \|\widehat{\pi}(\varphi_j(f))\widehat{\tau}(\xi) - \widehat{\tau}(\xi)\widehat{\pi}(\varphi_j(f))\| \\
&= \|\widehat{\tau}\left( \pi_\mathcal{H}(\varphi_j(f))\xi - \xi\cdot \varphi_j(f)\right)\| \\
&\leq \|\pi_\mathcal{H}(\varphi_j(f))\xi - \xi\cdot \varphi_j(f)\| < \epsilon.
\end{split}
\end{equation*}
\item If $x = \widehat{\tau}(\xi)^{\ast}$ for some $\xi \in F_2$, then the argument is similar after taking adjoints. Here, one uses the fact that $K = K^{\ast}$ and that positive linear maps preserve adjoints.
\end{itemize}
\item We wish to prove that $\|\beta_h(\psi_j(f)) - \psi_j(\sigma_h(f))\| < \epsilon$ for all $f\in K$ and $h\in G$. However, this follows from condition (vi) of \cref{definition: rokhlin_correspondence_compact} and the fact that
\[
\beta_h\circ \psi_j = \beta_h\circ \widehat{\pi}\circ \varphi_j = \widehat{\pi}\circ \alpha_h\circ \varphi_j
\]
by \cref{lem: action_pimsner}.
\end{enumerate}
Thus, we have verified the conditions of \cref{lem: rokhlin_local}, completing the proof.
\end{proof}

\section{Equivariant Bundles}\label{sec: bundles}

We are finally in a position to prove our main result. Once again, we begin with some definitions (see \cite[Chapter 1]{atiyah} for further details). \\

Throughout this section, we fix a compact metric space $X$. By a \emph{vector bundle} over $X$, we mean a triple $(E,p,X)$ where $p:E\to X$ is a locally trivial, complex vector bundle, endowed with a Hermitian metric, whose fibers (denoted by $\{E_x : x\in X\}$) are finite dimensional Hilbert spaces. Then, the space $\Gamma(E)$ of all continuous sections on $E$ is a finitely generated Hilbert $C(X)$-module, and carries a natural (central) action of $C(X)$. We write $(\Gamma(E), C(X), \pi_E)$ for this C*-correspondence, and $\mathcal{O}_E$ for the associated Cuntz-Pimsner algebra. \\

Let $G$ be a topological group. An action $(\widetilde{\alpha}, \widetilde{\gamma})$ of $G$ on $(E,p,X)$ is a pair of continuous actions $\widetilde{\alpha}: G\to \text{Homeo}(X)$ and $\widetilde{\gamma} : G\to \text{Homeo}(E)$ such that $p:E\to X$ is $G$-equivariant, and, for each $s\in G$ and $x\in X$, the map $E_x \to E_{\widetilde{\alpha}_s(x)}$ is a linear map of vector spaces that preserves the inner product. Given such a group action, there is a natural action $(\alpha, \gamma)$ of $G$ on the C*-correspondence $(\Gamma(E), C(X), \pi_E)$ given by
\begin{equation}\label{eqn: induced_action_correspondence}
\alpha_s(f)(x) := f(\widetilde{\alpha}_{s^{-1}}(x)), \text{ and } \gamma_s(\xi)(x) := \widetilde{\gamma}_s(\xi(\widetilde{\alpha}_{s^{-1}}(x))).
\end{equation}
Therefore, we get an induced action $\beta : G\to \text{Aut}(\mathcal{O}_E)$ on the associated Cuntz-Pimsner algebra by \cref{lem: action_pimsner}. \\

The following lemma is a partial generalization of \cite[Theorem 4.1]{gardella_compact}. There, the author assumes that $G$ is a compact Lie group (as the proof relies on the existence of local cross-sections). Here, we drop that assumption, but we need $X$ to be metrizable.

\begin{lemma}\label{lem: rokhlin_commutative}
Let $X$ be a compact metric space, and $\widetilde{\alpha}:G\to \text{Homeo}(X)$ be an action of a compact, second countable group on $X$. Write $\alpha : G\to \text{Aut}(C(X))$ for the induced action on $C(X)$. If $\widetilde{\alpha}$ is free, then
\[
\dim_{Rok}(\alpha) \leq \dim(X/G).
\]
\end{lemma}
\begin{proof}
If $\widetilde{\alpha}$ is free, then $Y := X/G$ is compact and Hausdorff. Furthermore, if $\pi : X\to Y$ denotes the quotient map, then $\pi^{\ast} : C(Y) \to C(X)$ gives $A := C(X)$ the structure of a $C(Y)$-algebra. Since $X$ is metrizable, $A$ is separable, so we may apply \autoref{thm: rokhlin_cx_algebra}. \\

To that end, observe that the action $\alpha : G\to \text{Aut}(A)$ is by $C(Y)$-linear automorphisms. Furthermore, for a point $y = \pi(x)\in Y$, the fiber of $A$ at $y$ is $C(G\cdot x)$. But the map $\Theta: G\to G\cdot x$ given by $g \mapsto \widetilde{\alpha}_g(x)$ induces a $G$-equivariant isomorphism $\Theta^{\ast}: (C(G\cdot x), \alpha_y) \to (C(G), \sigma)$. Since $\sigma$ has the Rokhlin property, it follows that
\[
\dim_{Rok}(\alpha_y) = 0.
\]
for all $y\in Y$. By \autoref{thm: rokhlin_cx_algebra}, we conclude that $\dim_{Rok}(\alpha) \leq \dim(Y)$. 
\end{proof}

In the statement of this lemma, if we further assume that $X$ is finite dimensional and that $G$ is a compact Lie group, then 
\[
\dim(X/G)\leq \dim(X) < \infty
\]
by \cite[Corollary 1.7.32]{palais}. Therefore, $\dim_{Rok}(\alpha) < \infty$. Of course, it may very well happen that $X/G$ is finite dimensional even if $G$ is not a Lie group. \\

We now prove the first part of \cref{mainthm: action_bundle}.

\begin{theorem}\label{thm: rokhlin_free}
Let $G$ be a compact, second countable group and $X$ a compact metric space. Let $(\widetilde{\alpha},\widetilde{\gamma})$ be an action of $G$ on a vector bundle $(E,p,X)$, and let $\beta : G\to \text{Aut}(\mathcal{O}_E)$ denote the induced action on the associate Cuntz-Pimsner algebra. If the action $\widetilde{\alpha}$ of $G$ on $X$ is free, then
\[
\dim_{Rok}(\beta) \leq \dim(X/G).
\]
\end{theorem}
\begin{proof}
We write $(\alpha,\gamma)$ for the action of $G$ on $(\Gamma(E), C(X),\pi_E)$ induced by $(\widetilde{\alpha}, \widetilde{\gamma})$ as above. Assume without loss of generality that $d := \dim(X/G) < \infty$. Then, by \cref{lem: rokhlin_commutative}, $\dim_{Rok}(\alpha) \leq d$. We claim that the pair $(\alpha,\gamma)$ is $d$-decomposable in the sense of \cref{definition: rokhlin_correspondence_compact}. To see this, fix finite sets $F_1 \subset C(X), F_2\subset \Gamma(E), K\subset C(G)$ and $\epsilon > 0$, and assume without loss of generality $1_{C(X)} \in F_1$ and that $\|\xi\|\leq 1$ for all $\xi \in F_2$. Since $\dim_{Rok}(\alpha) \leq d$, there exist $(d+1)$ c.c.p. maps
\[
\varphi_0, \varphi_1, \ldots, \varphi_d : C(G)\to C(X)
\]
forming a $(F_1,K,\epsilon)$-Rokhlin system as in \cref{lem: rokhlin_local}. Then, it is  clear that these maps satisfy conditions (i), (ii), (iv) and (vi) of \cref{definition: rokhlin_correspondence_compact}. Furthermore, it satisfies condition (v) because the action of $C(X)$ on $\Gamma(E)$ is central (in the sense that $\pi_E(f)\xi = \xi\cdot f$ for all $f\in C(X)$ and $\xi \in \Gamma(E)$). To verify condition (iii), set $x:= \sum_{j=0}^d \varphi_j(1_{C(G)})$, then, for any $\xi \in F_2$, we have
\[
\|\pi_E(x)\xi - \xi\| = \|\pi_E(x)\xi - \pi_E(1_{C(X)})\xi\| \leq \|x - 1_{C(X)}\| < \epsilon.
\]
The second part of condition (iii) once again holds because the action is central. So, by \cref{prop: rokhlin_pimsner}, we conclude that
\[
\dim_{Rok}(\beta) \leq d
\]
as required.
\end{proof}

Given a vector bundle $(E,p,X)$ as above, let  $(\widehat{\pi}, \widehat{\tau})$ denote the universal covariant representation of $(\Gamma(E), C(X), \pi_E)$ on $\mathcal{O}_E$. Since the action of $C(X)$ on $\Gamma(E)$ is central, for any $f\in C(X)$ and $\xi \in E$, one has
\[
\widehat{\pi}(f)\widehat{\tau}(\xi) = \widehat{\tau}(\pi_E(f)\xi) = \widehat{\tau}(\xi\cdot f) = \widehat{\tau}(\xi)\widehat{\pi}(f).
\]
Since $\mathcal{O}_E$ is generated by $\widehat{\pi}(C(X))\cup \widehat{\tau}(\Gamma(E))$, it follows that the map $\widehat{\pi} : C(X)\to \mathcal{O}_E$ gives $\mathcal{O}_E$ the structure of a $C(X)$-algebra.\\

Recall that a $C(X)$-algebra $A$ is said to be continuous if the maps $\{\Gamma_a : a\in A\}$ are all continuous. The next theorem (due to Vasselli) states that $\mathcal{O}_E$ is, in fact, a continuous $C(X)$-algebra, and also describes its fibers. Note that, for our purposes, all vector bundles have fibers that are finite dimensional Hilbert spaces. Recall that, for $n \in \N\cup\{+\infty\}$, the Cuntz algebra may be defined as $\mathcal{O}_n := C^{\ast}(s_1,s_2,\ldots, s_n)$ where $s_k \in \mathcal{B}(\ell^2)$ are isometries satisfying $\sum_{i=1}^n s_is_i^{\ast} = 1$.

\begin{theorem}\cite[Proposition 2]{vasselli}\label{thm: vasselli}
$\mathcal{O}_E$ is a continuous $C(X)$-algebra. Furthermore, if $n:X\to \Z$ denotes the rank function of $E$, then the fiber of $\mathcal{O}_E$ at a point $x\in X$ is $\mathcal{O}_{n(x)}$.
\end{theorem}

In order to apply \cref{thm: rokhlin_cx_algebra} in this setting, we need one more fact concerning certain actions of finite groups on Cuntz algebras. Recall that an automorphism $\alpha \in \text{Aut}(\mathcal{O}_n)$ is said to be \emph{outer} if there does not exist any unitary $u \in \mathcal{U}(\mathcal{O}_n)$ such that $\alpha(a) = uau^{\ast}$ for all $a\in \mathcal{O}_n$. Let $u = (u_{i,j}) \in \mathcal{U}_n(\C)$ be any non-trivial unitary matrix, and $\alpha_u \in \text{Aut}(\mathcal{O}_n)$ be the automorphism defined by
\[
\alpha_u(s_j) = \sum_{i=1}^n u_{i,j}s_i.
\]
Then, it is a result of Enomoto, Takehana and Watatani \cite{enomoto} that any such automorphism $\alpha_u$ is outer (see \cite{matsumoto} for an alternate proof). \\

Now, let $\rho : G\to \mathcal{U}_n(\C)$ be a faithful unitary representation of a compact group $G$, and let $\beta : G\to \text{Aut}(\mathcal{O}_n)$ be the induced action of $G$ on $\mathcal{O}_n$ given by \cref{lem: action_pimsner} (such actions are said to be \emph{quasi-free}). Then $\beta$ is pointwise outer. If $G$ is finite, then a result of Gardella \cite[Theorem 4.19]{gardella_compact} implies that
\[
\dim_{Rok}(\beta) \leq 1.
\]

This allows us to complete the proof of \cref{mainthm: action_bundle}. 

\begin{theorem}\label{thm: trivial_action_bundle}
Let $G$ be a finite group and $X$ a compact metric space. Let $(\widetilde{\alpha},\widetilde{\gamma})$ be an action of $G$ on a vector bundle $(E,p,X)$, and let $\beta : G\to \text{Aut}(\mathcal{O}_E)$ denote the induced action on the associated Cuntz-Pimsner algebra. If the action $\widetilde{\alpha}$ of $G$ on $X$ is trivial, and the representation of $G$ induced by $\widetilde{\gamma}$ on each fiber of $E$ is faithful, then
\[
\dim_{Rok}(\beta) \leq 2\dim(X)+1.
\]
\end{theorem}
\begin{proof}
Assume without loss of generality that $\dim(X) < \infty$. Since $X$ is metrizable, $C(X)$ is separable, and since $E$ is locally trivial, $\Gamma(E)$ is finitely generated (by Swan's theorem \cite{swan}). Hence, $\mathcal{O}_E$ is separable. Furthermore, $\mathcal{O}_E$ is nuclear by \cite[Corollary 7.4]{katsura}.\\

Now, let $(\alpha, \gamma)$ be the induced action of $G$ on the C*-correspondence $(\Gamma(E),C(X),\pi_E)$ given by \cref{eqn: induced_action_correspondence}. Since $\widetilde{\alpha}$ is trivial, $\alpha_s = \text{id}_{C(X)}$ for all $s\in G$, which implies (by \cref{lem: action_pimsner}) that $\beta_s\circ \widehat{\pi} = \widehat{\pi}$ for all $s\in G$. Therefore, $\beta$ acts by $C(X)$-linear automorphisms. \\

Now, for each $x\in X$, the representation $\gamma_x : G\to \text{Aut}(E_x)$ is assumed to be faithful. Hence, it follows from the above discussion that the induced quasi-free action $\beta_x : G\to \text{Aut}(\mathcal{O}_{n(x)})$ satisfies
\[
\dim_{Rok}(\beta_x)\leq 1.
\]
By \autoref{thm: rokhlin_cx_algebra}, we conclude that
\[
\dim_{Rok}(\beta) \leq (\dim(X)+1)(2) -1 = 2\dim(X) + 1.\qedhere
\]
\end{proof}

We now show how \cref{mainthm: action_bundle} may be applied to obtain structural properties of the corresponding crossed product. In what follows, we write $\dim_{nuc}(A)$ to denote the nuclear dimension of a C*-algebra $A$ \cite{nuclear_winter}.

\begin{corollary}
Let $X$ be a finite dimensional, compact metric space, and let $G$ be a compact, second countable group acting on a vector bundle $(E,p,X)$ as in \cref{mainthm: action_bundle}. Let $\beta : G\to \text{Aut}(\mathcal{O}_E)$ denote the induced action on the associated Cuntz-Pimsner algebra. Suppose that either of the following conditions hold:
\begin{enumerate}[label=(\roman*)]
\item The action of $G$ on $X$ is free, and $X/G$ is finite dimensional, or
\item $G$ is finite, the action of $G$ on $X$ is trivial, and the induced representation of $G$ on each fiber of $E$ is faithful.
\end{enumerate}
Then, $\mathcal{O}_E\rtimes_{\beta} G$ has finite nuclear dimension.
\end{corollary}
\begin{proof}
We know from \cite{enders} that $\dim_{nuc}(\mathcal{O}_n) = 1$ for all $n\in \N$. So, by \cref{thm: vasselli} and \cite[Theorem 4.13]{gardella}, we conclude that 
\[
\dim_{nuc}(\mathcal{O}_E) \leq (\dim(X)+1)(2) - 1 = 2\dim(X)+1.
\]
Under either of the conditions listed above, \cref{mainthm: action_bundle} implies that $\dim_{Rok}(\beta) < \infty$. Therefore, it follows from \cite[Theorem 3.4]{gardella_regularity} (see also \cite[Theorem 1.3]{hirshberg}) that
\[
\dim_{nuc}(\mathcal{O}_E\rtimes_{\beta} G) < \infty.\qedhere
\]
\end{proof}

Now, recall that finiteness of Rokhlin dimension depends on the existence of certain c.c.p. maps as described in \cref{defn: rokhlin_dim}. As mentioned in \cref{rem: commuting_ranges}, if one requires that these maps have commuting ranges, then one arrives at the definition of Rokhlin dimension \emph{with commuting towers}, and the associated integers is denoted by $\dim_{Rok}^c(\cdot)$. A closer inspection of the proofs leading up to \cref{thm: rokhlin_free} reveals the following: If the action of $G$ on $X$ is free, then
\[
\dim_{Rok}^c(\beta) \leq \dim(X/G).
\]
Therefore, if $X/G$ and $G$ are both finite dimensional, then \cite[Theorem D]{gardella} is applicable. This leads to the following permanence results.

\begin{corollary}
Let $X$ be a finite dimensional, compact metric space, and let $G$ be a finite dimensional, compact, second countable group acting on a vector bundle $(E,p,X)$ as in \cref{mainthm: action_bundle}. Let $\beta : G\to \text{Aut}(\mathcal{O}_E)$ denote the induced action on the associated Cuntz-Pimsner algebra. Suppose further that the action of $G$ on $X$ is free, and that $X/G$ is finite dimensional.
\begin{enumerate}[label=(\roman*)]
\item Let $\mathcal{D}$ denote a strongly self-absorbing C*-algebra such that $\mathcal{O}_n\otimes \mathcal{D} \cong \mathcal{O}_n$ for all $n\in \N$. Then, $\mathcal{O}_E\rtimes_{\beta} G$ absorbs $\mathcal{D}$ tensorially.
\item $\mathcal{O}_E\rtimes_{\beta} G$ has finite real rank.
\item $\mathcal{O}_E\rtimes_{\beta} G$ is nuclear, separable, and satisfies the UCT.
\end{enumerate}
\end{corollary}
\begin{proof}
Our hypotheses imply that $\dim^c_{Rok}(\beta) < \infty$ by \cref{mainthm: action_bundle}. By \cite[Theorem D]{gardella}, in each of the above situations, it suffices to show that $\mathcal{O}_E$ satisfies the required property.
\begin{enumerate}[label=(\roman*)]
\item Let $\mathcal{D}$ be a strongly self-absorbing C*-algebra such that $\mathcal{O}_n\otimes \mathcal{D}\cong \mathcal{O}_n$ for all $n\in \N$. As mentioned above, $\mathcal{O}_E$ is separable. Therefore, by \cref{thm: vasselli} and \cite[Theorem 4.6]{hirshberg_rordam_winter}, it follows that $\mathcal{O}_E$ absorbs $\mathcal{D}$ tensorially.
\item Since $\mathcal{O}_n$ has real rank zero, \cite[Theorem 4.13]{gardella} ensures that $\mathcal{O}_E$ has finite real rank.
\item Since $\mathcal{O}_n$ is nuclear, separable, and satisfies the UCT, the same is true for $\mathcal{O}_E$ by \cite[Theorem 1.4]{mdd_fiberwise}.\qedhere
\end{enumerate}
\end{proof}

As mentioned earlier, the condition that $X/G$ is finite dimensional holds if $G$ is assumed to be a compact Lie group by \cite[Corollary 1.7.32]{palais}. We now conclude the paper with an example to show that \cref{thm: trivial_action_bundle} does not extend to the case when $G$ is infinite. 

\begin{example}
Let $G = S^1$, and $\gamma : G\to \mathcal{U}_2(\C)$ be the representation
\[
\gamma_z = \begin{pmatrix}
z & 0 \\
0 & 1
\end{pmatrix}.
\]
Then the induced action $\beta : G\to \text{Aut}(\mathcal{O}_2)$ is given on the canonical generators of $\mathcal{O}_2$ by
\[
\beta_z(s_1) = zs_1 \text{ and } \beta_z(s_2) = s_2.
\]
It follows by an application of \cite[Theorem 4.4]{kishimoto} (see also \cite[Theorem 4.8]{katsura_ideal}) that the crossed product $\mathcal{O}_2\rtimes_{\beta} G$ is not simple (in the notation of \cite{katsura_ideal}, we have $\Gamma = \widehat{G} = \Z, \omega = (1,0) \in \Gamma^2$, and the set $\Z_{\geq 0} \subset \Gamma$ is a non-trivial $\omega$-invariant set). Hence, it follows from \cite[Corollary 2.17]{gardella} that $\beta$ cannot have finite Rokhlin dimension.
\end{example}

\textbf{Acknowledgements.} The author was partially supported by the SERB (Grant No. MTR/2020/000385). The author would like to thank the referee for a careful reading of the manuscript, and Eusebio Gardella for useful discussions had after this paper first appeared on the arXiv.

\end{document}